\documentclass[letterpaper, 10 pt, conference]{ieeeconf}  

\IEEEoverridecommandlockouts                              
\usepackage{algorithm}
\usepackage{algcompatible}
\usepackage{booktabs}
\usepackage{url}
\usepackage{multicol}

\usepackage{amsmath} 
\usepackage{amssymb}  

\title{\LARGE \bf
An Overflow Free Fixed-point Eigenvalue Decomposition Algorithm: Case Study of Dimensionality Reduction in Hyperspectral Images
}

\author{Bibek Kabi$^{1}$ and Anand S Sahadevan$^{2}$ and Tapan Pradhan$^{3}$ \\
 $^{1}$Laboratoire d'Informatique de l'Ecole Polytechnique, CNRS, Universit\'e Paris-Saclay, 91128 Palaiseau, France\\
 $^{2}$Space Application Center, Indian Space Research Organisation, India\\
 $^{3}$Techno India University, West Bengal, India\\
 $^{1}$bibek@lix.polytechnique.fr, $^{2}$anandss@sac.isro.gov.in, $^{3}$hod.ee.tiu@gmail.com
}

\begin{document}

\maketitle
\thispagestyle{empty}
\pagestyle{empty}

\begin{abstract}
We consider the problem of enabling robust range estimation of eigenvalue decomposition (EVD) algorithm for a reliable fixed-point design. The simplicity of fixed-point circuitry has always been so tempting to implement EVD algorithms in fixed-point arithmetic. Working towards an effective fixed-point design, integer bit-width allocation is a significant step which has a crucial impact on accuracy and hardware efficiency. This paper investigates the shortcomings of the existing range estimation methods while deriving bounds for the variables of the EVD algorithm. In light of the circumstances, we introduce a range estimation approach  based on vector and matrix norm properties together with a scaling procedure that maintains all the assets of an analytical method. The method could derive robust and tight bounds for the variables of EVD algorithm. The bounds derived using the proposed approach remain same for any input matrix and are also independent of the number of iterations or size of the problem. Some benchmark hyperspectral data sets have been used to evaluate the efficiency of the proposed technique. It was found that by the proposed range estimation approach, all the variables generated
during the computation of Jacobi EVD is bounded within $\pm1$.
\end{abstract}


\section{Introduction}
\label{sec:intro} 
Eigenvalue decomposition (EVD) is a key building block in signal processing and control applications. The fixed-point development of eigenvalue decomposition (EVD) algorithm have been extensively studied in the past few years \cite{zoran2007,chen2013reconfigurable,cordicsvdelsevier,naun,ieri,szecowka10,truncatedsvd,svdmimo,circuitandsystems,dacconf} because fixed-point circuitry is significantly simpler and faster. Owing to its simplicity, fixed-point arithmetic is ubiquitous in low cost embedded platforms. Fixed-point arithmetic has played an important role in supporting the field-programmable-gate-array (FPGA) parallelism by keeping the hardware resources as low as possible. The most crucial step involved in the float-to-fixed conversion process is deciding the integer wordlengths (IWLs) in order to avoid overflow. This step has a significant impact on accuracy and hardware resources. \\
\indent IWLs can be determined either using simulation \cite{zoran2007}, \cite{kim98}, \cite{ieee} or by analytical (formal) methods \cite{lee2006accuracy}, \cite{lopez2007improved}, \cite{Akyildiz-02}, \cite{384253}. Existing works on fixed-point EVD have mainly used simulation-based approach for finding the IWLs \cite{zoran2007,chen2013reconfigurable,cordicsvdelsevier,naun,ieri,szecowka10,truncatedsvd,svdmimo,circuitandsystems,dacconf,pradhan2016development} because of its capability to be performed on any kind of systems. In simulation-based methods, variable bounds are estimated using the extreme values obtained from the simulation of the floating-point model. This method needs a large amount of input matrices to obtain a reliable estimation of ranges. Thus, the method is quite slow. Moreover, it cannot guarantee to avoid overflow for non-simulated matrices. This is primarily due to the diverse range of input data. A stochastic range estimation method is discussed in \cite{6:3:380} which computes the ranges by propagating statistical distributions through operations. It requires large number of simulations to estimate system parameters and an appropriate input data set to estimate quality parameters \cite{constantinides2011numerical}, \cite{lopez2007improved} and therefore, it does not produce absolute bounds \cite{banciu2012stochastic}. \\
\indent There are several limitations associated with analytical (formal) methods. An analytical method based on $L_{1}$ norm and transfer function is described in \cite{carletta2003determining}. This method produces theoretical bounds that guarantee no overflow will occur, but the approach is only limited to linear time-invariant systems \cite{constantinides2011numerical}. Interval arithmetic (IA) ignores correlation among signals resulting in an overestimation of ranges \cite{lee2006accuracy}. Affine arithmetic (AA) is a preferable approach that takes into account the interdependency among the signals \cite{goubault2015zonotopic}, but ranges determined through AA explode during division if the range of divisor includes zero \cite{Akyildiz-02}, \cite{384253}. IA also suffers from the same problem. Both IA and AA are pessimistic approaches leading to higher implementation cost \cite{nehmeh2014integer}. Satisfiability modulo theory (SMT) produces tight ranges compared to IA and AA \cite{Akyildiz-02}, \cite{kinsman2011automated}. However, it is computationally expensive and much slower as compared to IA and AA \cite{384253}, \cite{7:1:137}. IA and AA methods compute the ranges of the intermediate variables by propagating the bounds of the input data through the arithmetic operations. SMT refines the range results provided by IA and AA. There are common issues associated with IA, AA and SMT. Given a particular range of the input matrix, these analytical methods compute certain ranges of the intermediate variables based on the arithmetic operations. However, if the range of the input matrix changes, the bounds for the variables no longer remain the same. Another issue with these analytical methods is that the bounds of the variables obtained using these methods are not independent of the number of iterations or the size of the problem. We exemplify these common issues associated with IA, AA and SMT in the next section. 
\section{Motivation}
In this section, we illustrate the issues associated with the existing range estimation methods through a motivational example (dimensionality reduction of hyperspectral images using fixed-point EVD). Along with the covariance matrices of hyperspectral images, we have also used some random symmetric positive semi-definite matrices generated from MATLAB. We have chosen such an instance because it is discovered from the literature that some of the works on dimensionality reduction of hyperspectral images highlight the overflow issues while using fixed-point EVD algorithm. A sincere effort has been made to contemplate them.  \\
\indent The diverse range of the elements of the input data matrices for different hyperspectral images (HSIs) limits the use of fixed-point EVD for dimensionality reduction \cite{hyperfpga3}, \cite{egho2012acceleration}. If the range of the input data is diverse, selecting a particular IWL may not avoid overflow for all range of input cases. Egho \emph{et al.}
\cite{egho2014adaptive} stated that fixed-point implementation of EVD algorithm leads to inaccurate computation of eigenvalues and eigenvectors due to overflow. Therefore, the authors implemented Jacobi algorithm in FPGA using floating-point
arithmetic. Lopez \emph{et al.} \cite{hyperfpga3} reported
overflow issues while computing EVD in fixed-point arithmetic.
Burger \emph{et al.} \cite{burger2011data} mentioned that while processing millions of HSI, numerical instability like overflow should be
avoided. Hence, determination of proper IWLs for
variables of fixed-point EVD algorithm in order to free it from
overflow for all range of input data remains a major research
issue.  \\
\indent The most widely used
algorithm for dimensionality reduction is principal component analysis (PCA). PCA requires computation of eigenvalues ($\lambda$) and
eigenvectors ($x$) given by
\begin{equation}
\label{eq:PCA}
A = X\Lambda {X^{\mathrm{T}}},
\end{equation}
where $A$ is the covariance matrix, $\Lambda$ is a
diagonal matrix containing the eigenvalues and the columns of $X$
contain the eigenvectors. $X$ is a new coordinate basis for the
image. There are several algorithms developed in the
literature for EVD of symmetric matrices \cite{datta2010numerical}, \cite{golub}, \cite{pradhan2013comparative}. Among all, two-sided Jacobi algorithm 
is most accurate and numerically stable \cite{datta2010numerical}, \cite{demmel1992jacobi}.
Most of the work attempted for dimensionality reduction via EVD uses two-sided Jacobi
algorithm \cite{egho2014adaptive}, \cite{egho2012hardware}, \cite{gonzaleznovel}.
The same algorithm is used in this paper for computing EVD of the
covariance matrix of the hyperspectral data. Apart from the accuracy and
stability of Jacobi algorithm, it also has high degree potential for parallelism, and hence can be implemented on FPGA
\cite{ieri}, \cite{szecowka10}.  In \cite{cordicsvdelsevier,ieri,naun,szecowka10,truncatedsvd,svdmimo,circuitandsystems,dacconf} this algorithm is implemented on FPGA with fixed-point arithmetic
to reduce power consumption and silicon area. However, in
all the works, fixed-point implementation of Jacobi algorithm uses
the simulation-based approach for estimating the ranges of
variables. It does not produce promising bounds (as discussed
earlier in section I).  \\
\label{sec:hestenes}  
\begin{algorithm}[ht!]
\caption{Two-sided Jacobi EVD algorithm}
\begin{algorithmic}[1]
\label{alg:hestenesalg}
\STATE $X=I$;
\FOR {$l=1$ to $n$}
\FOR {$i=1$ to $n$}
\FOR {$j=i+1$ to $n$}
\STATE $a=A(i,i)$; 
\STATE $b=A(j,j)$;
\STATE $c=A(i,j)=A(j,i)$;
\STATEx /* compute the Jacobi rotation which diagonalizes
 $\begin{pmatrix}
 A(i,i) & A(i,j)  \\
 A(j,i) & A(j,j)  \\
 \end{pmatrix}$ =$\begin{pmatrix}
 a & c  \\
 c & b  \\
 \end{pmatrix}$ */
\STATE $t=\frac{sign(\frac{b-a}{c})\cdot |c|}{|\frac{b-a}{2}|+\sqrt{c^{2}+{(\frac{b-a}{2})}^{2}}}$;
\STATE $cs=1/\sqrt{1+t^{2}}$; 
\STATE $sn=cs\cdot t$;
\STATEx /* update the 2$\times$2 submatrix */
\STATE $A(i,i)=a-c\cdot t$; 
\STATE $A(j,j)=b+c\cdot t$; 
\STATE $A(i,j)=A(j,i)=0$;
\STATEx /* update the rest of rows and columns $i$ and $j$ */
\FOR {$k=1$ to $n$} except $i$ and $j$
\STATE $tmp=A(i,k)$; 
\STATE $A(i,k)=cs\cdot tmp - sn\cdot A(j,k)$;
\STATE $A(j,k)=sn\cdot tmp + cs\cdot A(j,k)$;
\STATE $A(k,i)=A(i,k)$; 
\STATE $A(k,j)=A(j,k)$;
\ENDFOR
\STATEx /* update the eigenvector matrix $X$ */
\FOR {$k=1$ to $n$}
\STATE $tmp=X(k,i)$; 
\STATE $X(k,i)=cs\cdot tmp - sn\cdot X(k,j)$;
\STATE $X(k,j)=sn\cdot tmp + cs\cdot X(k,j)$;
\ENDFOR
\ENDFOR
\ENDFOR
\ENDFOR
\STATEx /* eigenvalues are diagonals of the final $A$ */
\FOR {$i=1$ to $n$}
\STATE $\lambda_{i}={A(i,i)}$;
\ENDFOR
\end{algorithmic}
\end{algorithm}
\indent Jacobi method computes EVD of a symmetric matrix $A$ by producing a
sequence of orthogonally similar matrices, which eventually
converges to a diagonal matrix \cite{datta2010numerical} given by
\begin{equation}
\Lambda=J^{\mathrm{T}}AJ,
\end{equation}
where $J$ is the Jacobi rotation and $\Lambda$ is a diagonal
matrix containing eigenvalues ($\lambda$). In each step, we
compute a Jacobi rotation with $J$ and update $A$ to $J^TAJ$,
where $J$ is chosen in such a way that two off-diagonal entries of
a $2\times2$ matrix of $A$ are set to zero. This is called
two-sided or classical Jacobi method. Algorithm~1 lists the
steps for Jacobi method. In order to investigate the challenges with fixed-point EVD algorithm,
we have used four different types of HSI collected by
the space-borne (Hyperion), air-borne (ROSIS and AVIRIS), handheld sensors (Landscape) and Synthetic (simulated EnMap). The selected Hyperion image subset contains the
Chilika Lake (latitude: 19.63 N - 19.68 N, longitude: 85.13 E -
85.18 E) and its catchment areas \cite{nasa}, \cite{kabioverflow}. ROSIS data was acquired during a
flight campaign at Pavia University, northern Italy \cite{pavia}. AVIRIS data was gathered by the AVIRIS sensor over the Indian Pines test site in North-western Indiana. Landscape data is obtained
from the database available from Stanford University \cite{skauli2013collection}. The simulated EnMap image subset contains the Maktesh Ramon, Israel (30.57 N, 34.83 E) \cite{segl2012eetes}. The sizes of the covariance matrix for the images are $120 \times 120$ for Hyperion, $103 \times 103$ for ROSIS, $200 \times 200$ for AVIRIS, $148 \times 148$ for Landscape and $244 \times 244$ for simulated EnMap. Out
of 120, 103, 200, 148 and 244 bands, only a certain number of bands are
sufficient for obtaining suitable information due to the large
correlation between adjacent bands. Hence, the dimension of the
image should be reduced to
decrease the redundancy in the data. The principal components (PCs) are
decided from the magnitudes of the eigenvalues. The numbers of PCs
which explain 99.0$\%$ variance are retained for the
reconstruction purpose. The following paragraph describes the shortcomings of the existing range estimation methods while computing bounds for EVD algorithm. \\
\indent Tables~\ref{tab:compare} and \ref{tab:compare1} shows the ranges obtained for Hyperion and ROSIS using simulation, IA, AA and the proposed method. The simulation-based range analysis is performed by feeding the floating-point algorithm with each input matrix separately and observing the data range. Notice that the ranges or the required IWLs (Table~\ref{tab:iwl1}) estimated using the simulation-based approach for Hyperion cannot avoid overflow in case of ROSIS. In other words, based on the ranges obtained using the simulation of Hyperion data one would allocate 24 bits to the integer part, but these number of bits cannot avoid overflow in case of ROSIS. Simulation-based method can only produce exact bounds for the simulated cases. Thus, simulation-based method is characterized by a need for stimuli, due to which it cannot be relied upon in practical implementations. In contrast, the static or analytical or formal methods like IA and AA which depends on the arithmetic operations always provide worst-case bounds so that no overflow occurs. However, the bounds are highly overestimated compared to the actual bounds produced by simulation-based method as shown in Tables~\ref{tab:compare} and \ref{tab:compare1}. This increases the hardware resources unnecessarily. \\
\indent In order to examine the range explosion problem of IA and AA, we computed the range of $A$ using IA and AA for random symmetric positive semi-definite matrices of different sizes generated from MATLAB. Table~\ref{tab:explosion} shows how the range of $A$ explodes when computed through IA and AA.  All the range estimation using IA and AA have been carried out using double precision floating-point format. According to the IEEE 754 double precision floating-point format, the maximum number that can be represented is in the order of $10^{308}$. It is noticed in Table~\ref{tab:explosion} that whenever the range is more than the maximum representable number, it is termed as infinity. It is apparent from the algorithm that variable $A$ is some or the other way related to the computation of all other variables. So, with the range of $A$ becoming infinity, the range of other variables also result in infinity as shown in Tables~\ref{tab:compare} and \ref{tab:compare1}. The range of variable $A$ goes unbounded because of the pessimistic nature of bounds produced by IA and AA. All the issues with existing range estimation methods are handled meticulously by the proposed method that produces unvarying or robust bounds while at the same time tightens the ranges. This is quite apparent from Tables~\ref{tab:compare} and \ref{tab:compare1}.  
\begin{table}[h]
\centering
\caption{Comparison Between The Ranges Computed By Simulation, Interval Arithmetic And Affine Arithmetic With Respect To The Proposed Approach For Hyperion Data With The Range Of The Covariance Matrix As $\lbrack \rm{-}2.42\rm{e-}05, 4.46\rm{e+}05 \rbrack$.} \label{tab:compare}
\begin{tabular}{p{0.4cm}|p{3.1cm}|p{1.0cm}|p{1.0cm}|p{1.0cm}}
  \hline
  Var & Simulation & IA & AA & Proposed\\
  \hline
    $A$ & $\lbrack \rm{-}1.02\rm{e+}06, 9.58\rm{e+}06 \rbrack$ & $\lbrack \rm{-}\infty, \infty \rbrack$ & $\lbrack \rm{-}\infty, \infty \rbrack$ & $\lbrack \rm{-1}, 1 \rbrack$ \\
  \hline
  $t$  & $\lbrack \rm{-1}, 1 \rbrack$ & $\lbrack \rm{-1}, 1 \rbrack$ & $\lbrack \rm{-1}, 1 \rbrack$ & $\lbrack \rm{-1}, 1 \rbrack$\\
  \hline
  $cs$  & $\lbrack 0.71, 1 \rbrack$ & $\lbrack 0, 1 \rbrack$ & $\lbrack 0, 1 \rbrack$ & $\lbrack 0, 1 \rbrack$ \\
  \hline
  $sn$  & $\lbrack \rm{-1}, 1 \rbrack$ & $\lbrack \rm{-1}, 1 \rbrack$  & $\lbrack \rm{-1}, 1 \rbrack$ & $\lbrack \rm{-1}, 1 \rbrack$\\
  \hline
  $a$  & $\lbrack 0, 9.58\rm{e+}06 \rbrack$ & $\lbrack \rm{-}\infty, \infty \rbrack$  & $\lbrack \rm{-}\infty, \infty \rbrack$  & $\lbrack 0, 1 \rbrack$ \\
  \hline
  $b$  & $\lbrack 0, 2.23\rm{e+}06 \rbrack$ & $\lbrack \rm{-}\infty, \infty \rbrack$  & $\lbrack \rm{-}\infty, \infty \rbrack$  & $\lbrack 0, 1 \rbrack$ \\
  \hline
  $c$  & $\lbrack \rm{-}1.02\rm{e+}06, 1.16\rm{e+}06 \rbrack$ & $\lbrack \rm{-}\infty, \infty \rbrack$ & $\lbrack \rm{-}\infty, \infty \rbrack$ & $\lbrack \rm{-1}, 1 \rbrack$ \\
  \hline
  $X$  & $\lbrack \rm{-}0.874, 1 \rbrack$ & $\lbrack \rm{-}\infty, \infty \rbrack$  & $\lbrack \rm{-}\infty, \infty \rbrack$  & $\lbrack \rm{-1}, 1 \rbrack$ \\
  \hline
  $\lambda$ & $\lbrack 6.47\rm{e-}10, 9.58\rm{e+}06 \rbrack$ & $\lbrack \rm{-}\infty, \infty \rbrack$ & $\lbrack \rm{-}\infty, \infty \rbrack$ & $\lbrack 0, 1 \rbrack$ \\
  \hline
\end{tabular}
\end{table}
\begin{table}[h]
\centering 
\caption{Comparison Between The Ranges Computed By Simulation, Interval Arithmetic and Affine Arithmetic With Respect To The Proposed Approach For ROSIS Data With The Range Of The Covariance Matrix As $\lbrack \rm{-}2.67\rm{e-}05, 5.81\rm{e+}05 \rbrack$.} \label{tab:compare1}
\begin{tabular}{p{0.4cm}|p{3.1cm}|p{1.0cm}|p{1.0cm}|p{1.0cm}}
  \hline
  Var & Simulation & IA & AA & Proposed\\
  \hline
    $A$ & $\lbrack \rm{-}3.27\rm{e+}06, 2.04\rm{e+}07 \rbrack$ & $\lbrack -\infty, \infty \rbrack$ & $\lbrack -\infty, \infty \rbrack$ & $\lbrack \rm{-1}, 1 \rbrack$ \\
  \hline
  $t$  & $\lbrack \rm{-1}, 1 \rbrack$ & $\lbrack \rm{-1}, 1 \rbrack$ & $\lbrack \rm{-1}, 1 \rbrack$ & $\lbrack \rm{-1}, 1 \rbrack$ \\
  \hline
  $cs$  & $\lbrack 0.71, 1 \rbrack$ & $\lbrack 0, 1 \rbrack$ & $\lbrack 0, 1 \rbrack$ & $\lbrack 0, 1 \rbrack$ \\
  \hline
  $sn$  & $\lbrack \rm{-1}, 1 \rbrack$ & $\lbrack \rm{-1}, 1 \rbrack$ & $\lbrack \rm{-1}, 1 \rbrack$ & $\lbrack \rm{-1}, 1 \rbrack$
 \\
  \hline
  $a$  & $\lbrack 0, 2.04\rm{e+}07 \rbrack$ & $\lbrack \rm{-}\infty, \infty \rbrack$ & $\lbrack \rm{-}\infty, \infty \rbrack$ & $\lbrack 0, 1 \rbrack$ \\
  \hline
  $b$  & $\lbrack 0, 2.51\rm{e+}06 \rbrack$ & $\lbrack \rm{-}\infty, \infty \rbrack$ & $\lbrack \rm{-}\infty, \infty \rbrack$ & $\lbrack 0, 1 \rbrack$ \\
  \hline
  $c$  & $\lbrack \rm{-}3.27\rm{e+}06, 2.13\rm{e+}06 \rbrack$ & $\lbrack \rm{-}\infty, \infty \rbrack$ & $\lbrack \rm{-}\infty, \infty \rbrack$ & $\lbrack \rm{-1}, 1 \rbrack$ \\
  \hline
  $X$ & $\lbrack \rm{-}0.768, 1 \rbrack$ & $\lbrack \rm{-}\infty, \infty \rbrack$  & $\lbrack \rm{-}\infty, \infty \rbrack$ & $\lbrack \rm{-1}, 1 \rbrack$ \\
  \hline
  $\lambda$ & $\lbrack 2.23\rm{e-}10, 2.04\rm{e+}07 \rbrack$ & $\lbrack \rm{-}\infty, \infty \rbrack$ & $\lbrack \rm{-}\infty, \infty \rbrack$ & $\lbrack 0, 1 \rbrack$ \\
  \hline
\end{tabular}
\end{table}
\begin{table}[ht!]
\centering 
\caption{Comparison Between The Integer Wordlengths Required Based On The Ranges Estimated By Simulation-Based Approach Shown In Tables~\ref{tab:compare} and \ref{tab:compare1}}\label{tab:iwl1}
\begin{tabular}{p{0.4cm}|p{0.9cm}|p{0.7cm}}
  \hline
  Var & Hyperion & ROSIS  \\
  \hline
    $A$ & 24 & 25  \\
  \hline
  $a$  & 22 & 22  \\
  \hline
  $b$  & 24 & 25  \\
  \hline
  $\lambda$ & 24 & 25  \\
  \hline
\end{tabular}
\end{table}
\begin{table*}
\centering
\caption{Range Explosion of $A$ While Computing Range Using Interval Arithmetic And Affine Arithmetic.}\label{tab:explosion}
\begin{tabular}{p{0.43cm}|p{1.15cm}|p{1.8cm}|p{3.0cm}|p{2.7cm}|p{2.7cm}|p{2.7cm}}
\hline
    Size &  Start & $l$=1, $i$=1, $j$=2 & $l$=3, $i$=3, $j$=4 & $l$=4, $i$=4, $j$=6 & $l$=6, $i$=6, $j$=8 & End \\ \hline
    $n$=2      & $\lbrack 0.65, 0.95 \rbrack$  & $\lbrack \rm{-}0.59, 2.35 \rbrack$ &  &  &  & $\lbrack \rm{-}4.06, 4.18  \rbrack$ \\
    \hline
    $n$=4      & $\lbrack 0.03, 0.93 \rbrack$ & $\lbrack \rm{-}5.52, 14.94 \rbrack$ & $\lbrack \rm{-}2.19\rm{e+}21, 2.20\rm{e+}21 \rbrack$ &  &  & $\lbrack \rm{-}3.7\rm{e+}28, 3.7\rm{e+}28 \rbrack$ \\
    \hline
    $n$=6 &  $\lbrack 0.18, 0.79 \rbrack$ & $\lbrack \rm{-}6.90, 28.62 \rbrack$ & $\lbrack \rm{-}6.96\rm{e+}61, 7.08\rm{e+}61 \rbrack$ & $\lbrack \rm{-}2.5\rm{e+}91, 2.6\rm{e+}91 \rbrack$ &  & $\lbrack \rm{-}4.5\rm{e+}139, 4.5\rm{e+}139 \rbrack$    \\
    \hline
    $n$=8 &  $\lbrack 0.11, 0.75 \rbrack$ & $\lbrack \rm{-}9.77, 48.29 \rbrack$ & $\lbrack \rm{-}1.01\rm{e+}126, 1.03\rm{e+}126 \rbrack$ & $\lbrack \rm{-}2.6\rm{e+}187, 2.6\rm{e+}187 \rbrack$ & $\lbrack \rm{-}1.6\rm{e+}301, 1.6\rm{e+}301 \rbrack$ & $\lbrack \rm{-}\infty, \infty \rbrack$    \\
    \hline
    $n$=10 &  $\lbrack 0.09, 0.96 \rbrack$ & $\lbrack \rm{-}16.62, 96.48 \rbrack$ & $\lbrack \rm{-}4.77\rm{e+}215, 4.88\rm{e+}215 \rbrack$ & $\lbrack \rm{-}\infty, \infty \rbrack$ & $\lbrack \rm{-}\infty, \infty \rbrack$ & $\lbrack \rm{-}\infty, \infty \rbrack$    \\
    \hline
    $n$=12 &  $\lbrack 0.03, 0.97 \rbrack$ & $\lbrack \rm{-}21.06, 139.76 \rbrack$ & $\lbrack \rm{-}\infty, \infty \rbrack$ & $\lbrack \rm{-}\infty, \infty \rbrack$ & $\lbrack \rm{-}\infty, \infty \rbrack$ & $\lbrack \rm{-}\infty, \infty \rbrack$    \\
    \hline

\end{tabular}
\end{table*}
In order to combat this, SMT has arisen which produce tight bounds compared to IA and AA. However, SMT is again computationally costly. Its runtime grows abruptly with application complexity. Hence, applying SMT for large size matrices would be too complex. Amidst the individual issues of the analytical methods, there are also some common issues. Provided with a particular range of the input matrix, the analytical methods (IA, AA and SMT) compute certain ranges of the intermediate variables based on the arithmetic operations. Notwithstanding, the ranges no longer remain the same, if the range of the input matrix changes. In order to investigate this issue, we consider two $2 \times 2$ symmetric input matrices given by   
\begin{equation}
 C=
  \begin{pmatrix}
0.4427    & 0.1067 \\
0.1067  &  0.4427  \\
  \end{pmatrix}
\end{equation} and 
\begin{equation}
 D=
  \begin{pmatrix}
33.4834    & 22.2054 \\
22.2054  &  33.4834  \\
  \end{pmatrix}.
\end{equation} 
The ranges obtained using IA and SMT in case of matrix $C$ cannot guarantee to avoid overflow in case of $D$ as shown in Tables~\ref{tab:cnd1} and \ref{tab:cnd2}. The fact is also similar for ranges derived using AA. This scenario is handled correctly by the proposed method that produces robust and tight bounds in both the cases $C$ and $D$. The range estimation using SMT was carried out using the freely available \emph{HySAT} implementation \cite{hysat}.\\
\indent There is one more common issue with these anaytical methods. We know that, provided with a fixed range of the input stimuli, these analytical (formal) methods successfully produce robust bounds \cite{kinsman2011automated}. Even though the range
of the input matrix is fixed, the bounds produced by these analytical methods would be robust only for a particular size of the problem or number of iterations. In other words, the bounds obtained will not be independent of the
number of iterations. 
\begin{table}[h]
\centering
\caption{Comparison Between The Ranges Computed By Simulation, Interval Arithmetic and Satisfiability-modulo-theory With Respect To The Proposed Approach For Input Matrix $C$.} \label{tab:cnd1}
\begin{tabular}{p{0.4cm}|p{1.8cm}|p{1.5cm}|p{1.5cm}|p{1.0cm}}
  \hline
  Var & Simulation & IA & SMT & Proposed \\
  \hline
    $C$ & $\lbrack 0, 0.549 \rbrack$ & $\lbrack \rm{-}3.88, 3.92 \rbrack$ & $\lbrack \rm{-}2.0, 3.2 \rbrack$ & $\lbrack \rm{-1}, 1 \rbrack$ \\
  \hline
  $t$  & $\lbrack \rm{-1}, 1 \rbrack$ & $\lbrack \rm{-}1, 1 \rbrack$ & $\lbrack \rm{-1}, 1 \rbrack$ & $\lbrack \rm{-1}, 1 \rbrack$ \\
  \hline
  $cs$  & $\lbrack 0, 1 \rbrack$ & $\lbrack 0, 1 \rbrack$ & $\lbrack 0, 1 \rbrack$ & $\lbrack 0, 1 \rbrack$ \\
  \hline
  $sn$  & $\lbrack \rm{-1}, 1 \rbrack$ & $\lbrack \rm{-1}, 1 \rbrack$  & $\lbrack \rm{-1}, 1 \rbrack$ & $\lbrack \rm{-1}, 1 \rbrack$\\
  \hline
  $a$  & $\lbrack 0, 0.336 \rbrack$ & $\lbrack 0, 0.336 \rbrack$  & $\lbrack 0, 0.336 \rbrack$  & $\lbrack 0, 1 \rbrack$ \\
  \hline
  $b$  & $\lbrack 0, 0.443 \rbrack$ & $\lbrack 0, 0.443 \rbrack$  & $\lbrack 0, 0.443 \rbrack$  & $\lbrack 0, 1 \rbrack$ \\
  \hline
  $c$  & 0 & 0 & 0 & $\lbrack \rm{-1}, 1 \rbrack$ \\
  \hline
  $X$  & $\lbrack \rm{-}0.707, 0.707 \rbrack$ & $\lbrack \rm{-}2.88, 2.88 \rbrack$  & $\lbrack \rm{-}1.76, 1.76 \rbrack$  & $\lbrack \rm{-1}, 1 \rbrack$ \\
  \hline
  $\lambda$ & $\lbrack 0.336, 0.549 \rbrack$ & $\lbrack \rm{-}2.29, 3.92 \rbrack$ & $\lbrack \rm{-}1.06, 3.22 \rbrack$ & $\lbrack 0, 1 \rbrack$ \\
  \hline
\end{tabular}
\end{table}
\begin{table}[h]
\centering
\caption{Comparison Between The Ranges Computed By Simulation, Interval Arithmetic And Satisfiability-modulo-theory With Respect To The Proposed Approach For Input Matrix $D$.} \label{tab:cnd2}
\begin{tabular}{p{0.3cm}|p{1.75cm}|p{2.0cm}|p{1.5cm}|p{0.85cm}}
  \hline
  Var & Simulation & IA & SMT & Proposed \\
  \hline
    $D$ & $\lbrack 0, 55.68 \rbrack$ & $\lbrack \rm{-}147.57, 151.96 \rbrack$ & $\lbrack \rm{-}87.2, 103.3 \rbrack$ & $\lbrack \rm{-1}, 1 \rbrack$ \\
  \hline
  $t$  & $\lbrack \rm{-1}, 1 \rbrack$ & $\lbrack \rm{-}1, 1 \rbrack$ & $\lbrack \rm{-1}, 1 \rbrack$ & $\lbrack \rm{-1}, 1 \rbrack$ \\
  \hline
  $cs$  & $\lbrack 0, 1 \rbrack$ & $\lbrack 0, 1 \rbrack$ & $\lbrack 0, 1 \rbrack$ & $\lbrack 0, 1 \rbrack$ \\
  \hline
  $sn$  & $\lbrack \rm{-1}, 1 \rbrack$ & $\lbrack \rm{-1}, 1 \rbrack$  & $\lbrack \rm{-1}, 1 \rbrack$ & $\lbrack \rm{-1}, 1 \rbrack$\\
  \hline
  $a$  & $\lbrack 0, 11.278 \rbrack$ & $\lbrack 0, 11.278 \rbrack$ & $\lbrack 0, 11.278 \rbrack$  & $\lbrack 0, 1 \rbrack$ \\
  \hline
  $b$  & $\lbrack 0, 33.483 \rbrack$ & $\lbrack 0, 33.483 \rbrack$  & $\lbrack 0, 33.483 \rbrack$  & $\lbrack 0, 1 \rbrack$ \\
  \hline
  $c$  & 0 & 0 & 0 & $\lbrack \rm{-1}, 1 \rbrack$ \\
  \hline
  $X$  & $\lbrack \rm{-}0.707, 0.707 \rbrack$ & $\lbrack \rm{-}2.36, 2.36 \rbrack$  & $\lbrack \rm{-}1.54, 1.54 \rbrack$  & $\lbrack \rm{-1}, 1 \rbrack$ \\
  \hline
  $\lambda$ & $\lbrack 11.278, 55.68 \rbrack$ & $\lbrack \rm{-}68.5, 151.96 \rbrack$ & $\lbrack \rm{-}12.6, 103.3 \rbrack$ & $\lbrack 0, 1 \rbrack$ \\
  \hline
\end{tabular}
\end{table}  
In order to illustrate this, let us consider two random symmetric positive definite matrices of sizes 3 $\times$ 3 and 5 $\times$ 5 given by
\begin{equation}
 Y=
  \begin{pmatrix}
46.7785 & 28.3501 &  18.8598 \\
28.3501 &  20.1805 &  13.0975 \\
18.8598 &  13.0975  &  8.6377  \\
  \end{pmatrix}
\end{equation} and
\begin{equation}
\begin{split}
 Z= \hspace{4cm} \\
  \begin{pmatrix}
107.6724 &  97.1687 & 107.1030 & 101.8092 &  78.4556 \\
97.1687 & 118.4738 & 109.0664 & 114.7589 & 101.8092 \\
107.1030 & 109.0664 & 126.1528 & 109.0664 & 107.1030 \\
101.8092 & 114.7589 & 109.0664 & 118.4738 &  97.1687  \\
78.4556 & 101.8092 & 107.1030 &  97.1687 & 107.6724 \\
  \end{pmatrix}.
\end{split}
\end{equation}
The bounds obtained using IA for the input matrices $Y$ and $Z$ are shown in Table~\ref{tab:cnd3}. The bounds are unnecessarily large compared to the actual bounds produced by the simulation-based approach shown in Table~\ref{tab:simulxy}. Now, the input matrices are scaled through the upper bound of their spectral norm to limit their range within $\rm{-}$1 and 1. The new matrices ${\hat Y}$ and ${\hat Z}$ whose elements range between $\rm{-}$1 and 1 are given by
\begin{equation}
 {\hat Y}=
  \begin{pmatrix}
0.2848  &  0.3945  &  0.3805 \\
0.3945  &   0.2848  &  0.3945 \\
0.3805  &  0.0163   & 0.2848  \\
  \end{pmatrix}
\end{equation} and
\begin{equation}
 {\hat Z}=
  \begin{pmatrix}
0.1160  &  0.2306  &  0.0349  &  0.3036  &  0.0860 \\
0.2306  &  0.1160  &  0.2306  &  0.0349  &  0.3036 \\
0.0349  &  0.2306  &  0.1160  &  0.2306  &  0.3435 \\
0.3036  &  0.0349  &  0.2306  &  0.1160  &  0.2306 \\
0.0860  &  0.3036  &  0.0349  &  0.2306  &  0.1160 \\
  \end{pmatrix}.
\end{equation}
The ranges obtained for the scaled input matrices are shown in Table~\ref{tab:cnd4}. Even though after scaling, the range of the variables obtained using IA are large and unbounded compared to the original bounds obtained using simulation-based method (Table~\ref{tab:cnd5}). The difference in unboundedness of the ranges shown in Tables~\ref{tab:cnd3} and \ref{tab:cnd4} is not substantially large. This illustrates that the ranges obtained using IA are not independent of the number of iterations. Similar is the case for both AA and SMT. We can observe the phenomenon in Table~\ref{tab:compare4} for one of the test hyperspectral data (simulated EnMAP). Inspite of the range of covariance matrix being $\lbrack \rm{-}3.71\rm{e-}06, 0.032 \rbrack$, the bounds estimated using IA and AA exploded compared to the actual bounds obtained using simulation-based approach. These examples comprehend that the bounds derived using the existing analytical methods are not independent of the number of iterations. Given the issues of the existing range estimation methods, our proposed method provides robust and tight bounds for the variables as shown in Tables~\ref{tab:compare}, \ref{tab:compare1}, \ref{tab:cnd1}, \ref{tab:cnd2} and \ref{tab:compare4}. Moreover, the bounds produced by the proposed method are independent of the size of the problem. The key to all these advantages is the usage of the scaling method and vector, matrix norm properties to derive the ranges.
\begin{table}
\centering
\caption{Ranges Computed By Interval Arithmetic For Input Matrices $Y$ and $Z$.} \label{tab:cnd3}
\begin{tabular}{p{1.0cm}|p{3.0cm}|p{3.0cm}}
  \hline
  Variables & IA ($Y$) & IA ($Z$)  \\
  \hline
    $Y$ or $Z$ & $\lbrack \rm{-}8.88\rm{e+}9, 8.94\rm{e+}9 \rbrack$ & $\lbrack \rm{-}4.51\rm{e+}71, 4.51\rm{e+}71 \rbrack$ \\
  \hline
  $t$  & $\lbrack \rm{-1}, 1 \rbrack$ & $\lbrack \rm{-}1, 1 \rbrack$ \\
  \hline
  $cs$  & $\lbrack 0, 1 \rbrack$ & $\lbrack 0, 1 \rbrack$  \\
  \hline
  $sn$  & $\lbrack \rm{-1}, 1 \rbrack$ & $\lbrack \rm{-1}, 1 \rbrack$  \\
  \hline
  $a$  & $\lbrack \rm{-}9.81\rm{e+}8, 9.93\rm{e+}8 \rbrack$ & $\lbrack \rm{-}1.80\rm{e+}70, 1.81\rm{e+}70 \rbrack$  \\
  \hline
  $b$  & $\lbrack \rm{-}9.81\rm{e+}8, 9.93\rm{e+}8 \rbrack$ & $\lbrack \rm{-}1.80\rm{e+}70, 1.81\rm{e+}70 \rbrack$   \\
  \hline
  $c$  & $\lbrack \rm{-}9.81\rm{e+}8, 9.93\rm{e+}8 \rbrack$ & $\lbrack \rm{-}1.80\rm{e+}70, 1.81\rm{e+}70 \rbrack$ \\
  \hline
  $X$  & $\lbrack \rm{-}9587, 10607 \rbrack$ & $\lbrack \rm{-}4.26\rm{e+}34, 4.38\rm{e+}34 \rbrack$  \\
  \hline
  $\lambda$ & $\lbrack \rm{-}8.88\rm{e+}9, 8.94\rm{e+}9 \rbrack$ & $\lbrack \rm{-}4.51\rm{e+}71, 4.51\rm{e+}71 \rbrack$  \\
  \hline
\end{tabular}
\end{table}
\begin{table}[ht!]
\centering
\caption{Ranges Computed By Simulation-Based Method For Input Matrices ${Y}$ and ${Z}$.} \label{tab:simulxy}
\begin{tabular}{p{1.0cm}|p{3.0cm}|p{3.0cm}}
  \hline
  Variables & Simulation (${Y}$) & Simulation (${Z}$)  \\
  \hline
    ${Y}$ or ${Z}$ & $\lbrack \rm{-}0.123, 72.98 \rbrack$ & $\lbrack \rm{-}15.73, 526.54 \rbrack$ \\
  \hline
  $t$  & $\lbrack \rm{-1}, 1 \rbrack$ & $\lbrack \rm{-}1, 1 \rbrack$  \\
  \hline
  $cs$  & $\lbrack 0, 1 \rbrack$ & $\lbrack 0, 1 \rbrack$  \\
  \hline
  $sn$  & $\lbrack \rm{-1}, 1 \rbrack$ & $\lbrack \rm{-1}, 1 \rbrack$  \\
  \hline
  $a$  & $\lbrack 0, 72.97 \rbrack$ & $\lbrack 0, 526.54 \rbrack$ \\
  \hline
  $b$  & $\lbrack 0, 20.18 \rbrack$ & $\lbrack 0, 526.54 \rbrack$   \\
  \hline
  $c$  & $\lbrack \rm{-}0.123, 28.35 \rbrack$ & $\lbrack \rm{-}8.45, 191.52 \rbrack$ \\
  \hline
  $X$  & $\lbrack \rm{-}0.61, 1 \rbrack$ & $\lbrack \rm{-}0.71, 1 \rbrack$ \\
  \hline
  $\lambda$ & $\lbrack 0.08, 72.98 \rbrack$ & $\lbrack 6.9\rm{e-}3, 526.54 \rbrack$  \\
  \hline
\end{tabular}
\end{table}
\begin{table}[ht!]
\centering
\caption{Ranges Computed By Interval Arithmetic For Input Matrices ${\hat Y}$ And ${\hat Z}$.} \label{tab:cnd4}
\begin{tabular}{p{1.0cm}|p{3.0cm}|p{3.0cm}}
  \hline
  Variables & IA (${\hat Y}$) & IA (${\hat Z}$)  \\
  \hline
    ${\hat Y}$ or ${\hat Z}$ & $\lbrack \rm{-}9.44\rm{e+}7, 9.51\rm{e+}7 \rbrack$ & $\lbrack \rm{-}8.08\rm{e+}68, 8.08\rm{e+}68 \rbrack$ \\
  \hline
  $t$  & $\lbrack \rm{-1}, 1 \rbrack$ & $\lbrack \rm{-}1, 1 \rbrack$  \\
  \hline
  $cs$  & $\lbrack 0, 1 \rbrack$ & $\lbrack 0, 1 \rbrack$  \\
  \hline
  $sn$  & $\lbrack \rm{-1}, 1 \rbrack$ & $\lbrack \rm{-1}, 1 \rbrack$  \\
  \hline
  $a$  & $\lbrack \rm{-}1.04\rm{e+}7, 1.06\rm{e+}7 \rbrack$ & $\lbrack \rm{-}3.23\rm{e+}67, 3.23\rm{e+}67 \rbrack$ \\
  \hline
  $b$  & $\lbrack \rm{-}1.04\rm{e+}7, 1.06\rm{e+}7 \rbrack$ & $\lbrack \rm{-}3.23\rm{e+}67, 3.23\rm{e+}67 \rbrack$   \\
  \hline
  $c$  & $\lbrack \rm{-}1.04\rm{e+}7, 1.06\rm{e+}7 \rbrack$ & $\lbrack \rm{-}3.23\rm{e+}67, 3.23\rm{e+}67 \rbrack$ \\
  \hline
  $X$  & $\lbrack \rm{-}9587, 10607 \rbrack$ & $\lbrack \rm{-}4.26\rm{e+}34, 4.37\rm{e+}34 \rbrack$ \\
  \hline
  $\lambda$ & $\lbrack \rm{-}9.44\rm{e+}7, 9.51\rm{e+}7 \rbrack$ & $\lbrack \rm{-}8.08\rm{e+}68, 8.08\rm{e+}68 \rbrack$  \\
  \hline
\end{tabular}
\end{table}
\begin{table}[ht!]
\centering
\caption{Ranges Computed By Simulation-Based Method For Input Matrices ${\hat Y}$ And ${\hat Z}$.} \label{tab:cnd5}
\begin{tabular}{p{1.0cm}|p{3.0cm}|p{3.0cm}}
  \hline
  Variables & Simulation (${\hat Y}$) & Simulation (${\hat Z}$)  \\
  \hline
    ${\hat Y}$ or ${\hat Z}$ & $\lbrack \rm{-}1.31\rm{e-}3, 0.78 \rbrack$ & $\lbrack \rm{-}0.028, 0.94 \rbrack$ \\
  \hline
  $t$  & $\lbrack \rm{-1}, 1 \rbrack$ & $\lbrack \rm{-}1, 1 \rbrack$  \\
  \hline
  $cs$  & $\lbrack 0, 1 \rbrack$ & $\lbrack 0, 1 \rbrack$  \\
  \hline
  $sn$  & $\lbrack \rm{-1}, 1 \rbrack$ & $\lbrack \rm{-1}, 1 \rbrack$  \\
  \hline
  $a$  & $\lbrack 0, 0.78 \rbrack$ & $\lbrack 0, 0.94 \rbrack$ \\
  \hline
  $b$  & $\lbrack 0, 0.21 \rbrack$ & $\lbrack 0, 0.94 \rbrack$   \\
  \hline
  $c$  & $\lbrack \rm{-}1.31\rm{e-}3, 0.31 \rbrack$ & $\lbrack \rm{-}0.015, 0.34 \rbrack$ \\
  \hline
  $X$  & $\lbrack \rm{-}0.61, 1 \rbrack$ & $\lbrack \rm{-}0.71, 1 \rbrack$ \\
  \hline
  $\lambda$ & $\lbrack 8.59\rm{e-}04, 0.78 \rbrack$ & $\lbrack 1.24\rm{e-}5, 0.94 \rbrack$  \\
  \hline
\end{tabular}
\end{table}
\begin{table}[ht!]
\centering 
\caption{Comparison Between The Ranges Computed By Simulation, Interval Arithmetic And Affine Arithmetic With Respect To The Proposed Approach For Simulated EnMAP Data With The Range Of The Covariance Matrix As $\lbrack \rm{-}3.71\rm{e-}06, 0.032 \rbrack$.}\label{tab:compare4}
\begin{tabular}{p{0.4cm}|p{3.1cm}|p{1.0cm}|p{1.0cm}|p{1.0cm}}
  \hline
  Var & Simulation & IA & AA & Proposed\\
  \hline
    $A$ & $\lbrack \rm{-}0.072, 1.29 \rbrack$ & $\lbrack -\infty, \infty \rbrack$ & $\lbrack -\infty, \infty \rbrack$ & $\lbrack \rm{-1}, 1 \rbrack$ \\
  \hline
  $t$  & $\lbrack \rm{-1}, 1 \rbrack$ & $\lbrack \rm{-1}, 1 \rbrack$ & $\lbrack \rm{-1}, 1 \rbrack$ & $\lbrack \rm{-1}, 1 \rbrack$ \\
  \hline
  $cs$  & $\lbrack 0.71, 1 \rbrack$ & $\lbrack 0, 1 \rbrack$ & $\lbrack 0, 1 \rbrack$ & $\lbrack 0, 1 \rbrack$ \\
  \hline
  $sn$  & $\lbrack \rm{-1}, 1 \rbrack$ & $\lbrack \rm{-1}, 1 \rbrack$ & $\lbrack \rm{-1}, 1 \rbrack$ & $\lbrack \rm{-1}, 1 \rbrack$
 \\
  \hline
  $a$  & $\lbrack 0, 1.29 \rbrack$ & $\lbrack \rm{-}\infty, \infty \rbrack$ & $\lbrack \rm{-}\infty, \infty \rbrack$ & $\lbrack 0, 1 \rbrack$ \\
  \hline
  $b$  & $\lbrack 0, 1.29 \rbrack$ & $\lbrack \rm{-}\infty, \infty \rbrack$ & $\lbrack \rm{-}\infty, \infty \rbrack$ & $\lbrack 0, 1 \rbrack$ \\
  \hline
  $c$  & $\lbrack \rm{-}0.067, 0.174 \rbrack$ & $\lbrack \rm{-}\infty, \infty \rbrack$ & $\lbrack \rm{-}\infty, \infty \rbrack$ & $\lbrack \rm{-1}, 1 \rbrack$ \\
  \hline
  $X$ & $\lbrack \rm{-}0.823, 1 \rbrack$ & $\lbrack \rm{-}\infty, \infty \rbrack$  & $\lbrack \rm{-}\infty, \infty \rbrack$ & $\lbrack \rm{-1}, 1 \rbrack$ \\
  \hline
  $\lambda$ & $\lbrack 1.24\rm{e-}05, 0.942 \rbrack$ & $\lbrack \rm{-}\infty, \infty \rbrack$ & $\lbrack \rm{-}\infty, \infty \rbrack$ & $\lbrack 0, 1 \rbrack$ \\
  \hline
\end{tabular}
\end{table}
\section{Proposed Solution}
\label{sec:scaling}
Particularizing, there are mainly three issues associated with the existing range estimation methods:
\begin{enumerate}
\item incompetence of the simulation-based approach to produce unvarying or robust bounds,
\item bounds produced by existing analytical (formal) methods are not independent of the number of iterations or size of the problem, and
\item overestimated bounds produced by IA and AA.
\end{enumerate}
Taking into account the issues 1 and 2, we propose in this study, an analytical method based on vector and matrix norm
properties to derive unvarying or robust bounds for the variables of EVD algorithm. The proof for deriving the bounds make use of the fact that
all the eigenvalues of a symmetric semi-positive definite matrix are bounded within the upper bound for the spectral norm of the matrix. Further taking into consideration the issue 3, we demonstrate that if the spectral norm of any matrix is kept within unity, tight ranges for the variables of the EVD algorithm can be derived. It is well-known that the spectral norm of any matrix is bounded by \cite{golub}, \cite{higham}
\begin{equation}
\label{eq:scaling}
{{\|A\|}_2} \leq \sqrt{{\|A\|}_1{\|A\|}_\infty}.
\end{equation}
For symmetric matrices, the spectral norm
${\|A\|}_2$ in (\ref{eq:scaling}) can be replaced with the
spectral radius ${\rho (A)}$.  
\newtheorem{theorem}{Theorem}
\begin{theorem}
Given the bounds for spectral norm as ${{\|A\|}_2} \leq \sqrt{{\|A\|}_1{\|A\|}_\infty}$, the Jacobi EVD algorithm applied to ${A}$ has the following bounds for the variables for all $i$, $j$, $k$ and $l$:
\begin{itemize}
\item $\lbrack {A}\rbrack_{kl}$ $\in$ $\lbrack-\sqrt{{\|A\|}_1{\|A\|}_\infty},\sqrt{{\|A\|}_1{\|A\|}_\infty}\rbrack$
\item $t$ $\in$ $\lbrack-1,1\rbrack$ 
\item $cs$ $\in$ $\lbrack0,1\rbrack$ 
\item $sn$ $\in$ $\lbrack-1,1\rbrack$ 
\item $\lbrack X \rbrack_{kl}$ $\in$ $\lbrack-1,1\rbrack$ 
\item $a$ $\in$ $\lbrack 0,\sqrt{{\|A\|}_1{\|A\|}_\infty}\rbrack$ 
\item $b$ $\in$ $\lbrack 0,\sqrt{{\|A\|}_1{\|A\|}_\infty}\rbrack$ 
\item $c$ $\in$ $\lbrack-\sqrt{{\|A\|}_1{\|A\|}_\infty},\sqrt{{\|A\|}_1{\|A\|}_\infty}\rbrack$ 
\item $\lbrack {\lambda_{i}} \rbrack_{k}$ $\in$ $ \lbrack 0,\sqrt{{\|A\|}_1{\|A\|}_\infty} \rbrack$ 
\end{itemize}
where $i$, $j$ denote the iteration number and $\lbrack \rbrack_{k}$ and $\lbrack \rbrack_{kl}$ denote the $k^{th}$ component of a vector and $kl^{th}$ component of a matrix respectively.
\end{theorem}
\begin{proof}
Using vector and matrix norm properties the ranges of the variables can be derived. We start by bounding the elements of the input symmetric matrix as
\begin{equation}
\label{eq:derive1}
\max\limits_{kl} |\lbrack {A}\rbrack_{kl}| \leq {\|{A}\|}_2 = \rho({A}) \leq \sqrt{{\|A\|}_1{\|A\|}_\infty},
\end{equation}
where (\ref{eq:derive1}) follows from \cite{golub}. Hence, the
elements of ${A}$ are in the range $\lbrack-\sqrt{{\|A\|}_1{\|A\|}_\infty},\sqrt{{\|A\|}_1{\|A\|}_\infty}\rbrack$. Line
30 in Algorithm~1 shows the computation of eigenvalues. We know
that $\rho({A}) \leq \sqrt{{\|A\|}_1{\|A\|}_\infty}$, so the upper bound for the eigenvalues is equal to $\sqrt{{\|A\|}_1{\|A\|}_\infty}$. In this
work, the fixed-point Jacobi EVD algorithm is applied to
covariance matrices. Due to the positive semi-definiteness
property of covariance matrices, the lower bound for the
eigenvalues is equal to zero. Thus, the range of ${\lambda_{i}}$ is $\lbrack0,\sqrt{{\|A\|}_1{\|A\|}_\infty}\rbrack$. The eigenvalues in Line 30
can also be calculated as
\begin{equation}
\label{eq:a1}
{\lambda_{i}}={\|{A}(:,i)\|}_2.
\end{equation}
According to vector norm property we can say that
\begin{equation}
\label{eq:a2}
{\|{A}(:,i)\|}_\infty \leq {\|{A}(:,i)\|}_2,
\end{equation}
where ${\|{A}(:,i)\|}_\infty$ is the maximum of the absolute
of the elements in ${A}(:,i)$. From the upper bound of ${\lambda_{i}}$, (\ref{eq:a1}) and (\ref{eq:a2}) we can say that
each element of ${A}(:,i)$ lie in the range
$\lbrack-\sqrt{{\|A\|}_1{\|A\|}_\infty},\sqrt{{\|A\|}_1{\|A\|}_\infty}\rbrack$. Thus all elements of ${A}$ lie in the
range $\lbrack-\sqrt{{\|A\|}_1{\|A\|}_\infty},\sqrt{{\|A\|}_1{\|A\|}_\infty}\rbrack$ for all the iterations. Since we
have considered symmetric positive semi-definite matrices (unlike
the off-diagonal entries the diagonal elements are always
positive), the diagonal elements of ${A}$ are in the range
$\lbrack0,\sqrt{{\|A\|}_1{\|A\|}_\infty}\rbrack$. Rest of the elements lie in the range
$\lbrack-\sqrt{{\|A\|}_1{\|A\|}_\infty},\sqrt{{\|A\|}_1{\|A\|}_\infty}\rbrack$. Line 5, 6 and 7 in Algorithm~1 computes $a$,
$b$ and $c$ respectively. Since $a$ and $b$ are the diagonal
elements of ${A}$, their range is $\lbrack0,\sqrt{{\|A\|}_1{\|A\|}_\infty}\rbrack$. $c$ is
the off-diagonal entry of ${A}$, therefore its range is
$\lbrack-\sqrt{{\|A\|}_1{\|A\|}_\infty},\sqrt{{\|A\|}_1{\|A\|}_\infty}\rbrack$. Line 8 in Algorithm~1 computes $t$. Let $t$=$r/s$ such that
\begin{equation}
\label{eq:r} 
r=sign \Big(\frac{b-a}{c} \Big)\cdot |c| \hspace{1 mm} \textrm{and} \hspace{1 mm} s={\Big|\frac{b-a}{2}\Big|+\sqrt{c^{2}+{\Big(\frac{b-a}{2}
\Big)}^{2}}}.
\end{equation}
According to (\ref{eq:r}), numerator ($r$) of $t$ lies in the range $\lbrack-|c|,|c|\rbrack$.
${|\frac{b-a}{2}|}$ and $\sqrt{c^{2}+{(\frac{b-a}{2})}^{2}}$ are always positive.
The summation $s$ is greater than or equal to $|c|$, because if $b=a$ then $s$ is equal to $c$ or if $b\ne a$ then $s$ is greater than $c$ since $\Big|\frac{b-a}{2}\Big|$ is greater than or equal to zero and $\sqrt{c^{2}+{(\frac{b-a}{2})}^{2}}$ is greater than $|c|$. From the range of $a$, $b$ and the denominator of $t$, we can say that $|c|$ will always be less or equal to ${|\frac{b-a}{2}|+\sqrt{c^{2}+{(\frac{b-a}{2})}^{2}}}$.
Thus, we can conclude that $t$ lies in the range
$\lbrack-1,1\rbrack$ and arc tangent of $t$ is
limited in the range $\lbrack-\frac{\pi}{4},\frac{\pi}{4}\rbrack$.
The Jacobi EVD method tries to make the off-diagonal entries of
2$\times$2 submatrix of $A$ zero by overwriting $A$ with $J^TAJ$.
According to 2$\times$2 symmetric Schur decomposition discussed in
\cite{golub}, $cs$ and $sn$ are cosine and sine trigonometric
functions. Thus, the bounds of $cs$ and $sn$ are
$\lbrack-1,1\rbrack$. Line 9 in Algorithm~1 computes $cs$ which
involves square root operation and therefore the range of $cs$ can
be modified to $\lbrack0,1\rbrack$. As the range of $cs$ and $t$
are $\lbrack0,1\rbrack$ and $\lbrack-1,1\rbrack$ respectively,
using multiplication rule of interval arithmetic
\cite{Hollis:1999:VBD:519964} the range of $sn$ (Line 10) can be
derived as $\lbrack-1,1\rbrack$. Next we bound the elements of
$X$. $X$ is the eigenvector matrix each column of which has unity
norm (eigenvectors of symmetric matrices are orthogonal). Hence all elements of $X$ are in the range
$\lbrack-1,1\rbrack$ following (\ref{eq:q}).
\begin{equation}
\label{eq:q}
{\|{X(:,i)}\|}_\infty \leq {\|{X(:,i)}\|}_2 = 1.
\end{equation}
Since the range of $A$ is $\lbrack-\sqrt{{\|A\|}_1{\|A\|}_\infty},\sqrt{{\|A\|}_1{\|A\|}_\infty}\rbrack$, according
to Line 15 of Algorithm~1, the range of $tmp$ can be fixed
as $\lbrack-\sqrt{{\|A\|}_1{\|A\|}_\infty},\sqrt{{\|A\|}_1{\|A\|}_\infty}\rbrack$.
\end{proof}
The bounds obtained according to Theorem 1 remain unchanged for all the iterations of the algorithm. The bounds are independent of the number of iterations or the size of the input matrix. Thus, the issue 2 has been handled accurately. Now considering the issue 1, the bounds according to Theorem 1 remain same (the pattern remains the same as shown in Theorem 1) for any input matrix, but depend on the factor $\sqrt{{\|A\|}_1{\|A\|}_\infty}$. For different input matrices, the magnitude of $\sqrt{{\|A\|}_1{\|A\|}_\infty}$ will change and this, in turn, will differ the bounds. The issue 1 has not yet been handled prudently. Hence, we propose that if the input matrix is scaled through $m=\sqrt{{\|A\|}_1{\|A\|}_\infty}$ then we can achieve a two-fold advantage: unvarying and tight bounds (solution for issue 3). This will resolve all the issues. If the input matrix is scaled as 
${\hat A}=\frac{A}{m}$, the EVD of matrix ${\hat A}$ is given as
\begin{equation}
{\hat A}x={\hat \lambda}x,
\end{equation}
where $Ax={\lambda} x$ and ${\hat \lambda}= \frac{\lambda}{m}$.
$x$ is the eigenvector and ${\lambda}$ is the eigenvalue. After
scaling through a scalar value, the original eigenvectors do not
change. The original eigenvalues change by a factor $\frac{1}{m}$. We need not recover the original eigenvalues because, in PCA, eigenvaues are only used to calculate the required number of PCs. Since, all the eigenvalues are scaled by the same factor, the number of PCs do not change whether the number is fixed using original eigenvalues or scaled ones. In applications, where original eigenvalues are required, the number of IWLs required is $\left\lceil\log_{2}(\sqrt{{\|A\|}_1{\|A\|}_\infty})\right\rceil$  depending on the magnitude of the scaling factor. Only the binary point of the eigenvalues is required to be adjusted online while for other variables it is fixed irrespective of the property of the input matrix.
\begin{theorem}
Given the scaling factor as $m=\sqrt{{\|A\|}_1{\|A\|}_\infty}$,
the Jacobi EVD algorithm (Algorithm~1) applied to ${\hat A}$ has the following bounds
for the variables for all $i$, $j$, $k$ and $l$: 
\begin{itemize}
\item $\lbrack {\hat A}\rbrack_{kl}$ $\in$ $\lbrack-1,1\rbrack$
\item $t$ $\in$ $\lbrack-1,1\rbrack$ 
\item $cs$ $\in$ $\lbrack0,1\rbrack$ 
\item $sn$ $\in$ $\lbrack-1,1\rbrack$ 
\item $\lbrack X \rbrack_{kl}$ $\in$ $\lbrack-1,1\rbrack$ 
\item $a$ $\in$ $\lbrack0,1\rbrack$ 
\item $b$ $\in$ $\lbrack0,1\rbrack$ 
\item $c$ $\in$ $\lbrack-1,1\rbrack$ 
\item $\lbrack {\hat \lambda_{i}} \rbrack_{k}$ $\in$ $ \lbrack 0,1 \rbrack$ 
\end{itemize}
where $i$, $j$ denote the iteration number and $\lbrack \rbrack_{k}$ and $\lbrack \rbrack_{kl}$ denote the $k^{th}$ component of a vector and $kl^{th}$ component of a matrix respectively.
\end{theorem}
\begin{proof}
Using vector and matrix norm properties the ranges of the variables can be derived. We start by bounding the elements of the input symmetric matrix as
\begin{equation}
\label{eq:derive2}
\max\limits_{kl} |\lbrack {\hat A}\rbrack_{kl}| \leq {\|{\hat A}\|}_2 = \rho({\hat A}) \leq 1,
\end{equation}
where (\ref{eq:derive2}) follows from \cite{golub}. Hence, the
elements of ${\hat A}$ are in the range $\lbrack-1,1\rbrack$. The remaining bounds are derived in the similar fashion as decribed in proof for Therorem 1. \end{proof}
The bounds on the variables of EVD algorithm obtained after scaling remain constant for all the iterations and also do not vary for any input
matrix. Besides, the bounds are also tight.  
\section{More data sets with results}
\label{sec:res}
In this section, we present a few more hyperspectral data sets, and we compare the bounds on variables of Jacobi EVD algorithm produced by the existing range estimation methods and the proposed approach. Tables~\ref{tab:compare2} and \ref{tab:compare3} show the comparison between the bounds on the variables obtained by existing range estimation methods with respect to the proposed approach through the AVIRIS and Landscape data sets. We can observe that the ranges estimated using the simulation for AVIRIS data cannot avoid overflow in case of Landscape. This is quite apparent from the number of integer bits required, shown in Table~\ref{tab:iwl2}. As usual, the bounds produced by IA and AA outbursted. However, the proposed method produces robust and tight bounds. The bounds obtained are independent of any range of the input matrix and also the number of iterations.  \\
\begin{table}[ht!]
\centering 
\caption{Comparison Between The Ranges Computed By Simulation, Interval Arithmetic And Affine Arithmetic With Respect To The Proposed Approach For AVIRIS Data With The Range Of The Covariance Matrix As $\lbrack \rm{-}5.01\rm{e+}05, 1.07\rm{e+}06 \rbrack$.} \label{tab:compare2}
\begin{tabular}{p{0.4cm}|p{3.1cm}|p{1.0cm}|p{1.0cm}|p{1.0cm}}
  \hline
  Var & Simulation & IA & AA & Proposed\\
  \hline
    $A$ & $\lbrack \rm{-}2.66\rm{e+}6, 2.68\rm{e+}07 \rbrack$ & $\lbrack -\infty, \infty \rbrack$ & $\lbrack -\infty, \infty \rbrack$ & $\lbrack \rm{-1}, 1 \rbrack$ \\
  \hline
  $t$  & $\lbrack \rm{-1}, 1 \rbrack$ & $\lbrack \rm{-1}, 1 \rbrack$ & $\lbrack \rm{-1}, 1 \rbrack$ & $\lbrack \rm{-1}, 1 \rbrack$ \\
  \hline
  $cs$  & $\lbrack 0.71, 1 \rbrack$ & $\lbrack 0, 1 \rbrack$ & $\lbrack 0, 1 \rbrack$ & $\lbrack 0, 1 \rbrack$ \\
  \hline
  $sn$  & $\lbrack \rm{-1}, 1 \rbrack$ & $\lbrack \rm{-1}, 1 \rbrack$ & $\lbrack \rm{-1}, 1 \rbrack$ & $\lbrack \rm{-1}, 1 \rbrack$
 \\
  \hline
  $a$  & $\lbrack 0, 2.68\rm{e+}07 \rbrack$ & $\lbrack \rm{-}\infty, \infty \rbrack$ & $\lbrack \rm{-}\infty, \infty \rbrack$ & $\lbrack 0, 1 \rbrack$ \\
  \hline
  $b$  & $\lbrack 0, 9.21\rm{e+}06 \rbrack$ & $\lbrack \rm{-}\infty, \infty \rbrack$ & $\lbrack \rm{-}\infty, \infty \rbrack$ & $\lbrack 0, 1 \rbrack$ \\
  \hline
  $c$  & $\lbrack \rm{-}2.38\rm{e+}06, 3.66\rm{e+}06 \rbrack$ & $\lbrack \rm{-}\infty, \infty \rbrack$ & $\lbrack \rm{-}\infty, \infty \rbrack$ & $\lbrack \rm{-1}, 1 \rbrack$ \\
  \hline
  $X$ & $\lbrack \rm{-}0.939, 1 \rbrack$ & $\lbrack \rm{-}\infty, \infty \rbrack$  & $\lbrack \rm{-}\infty, \infty \rbrack$ & $\lbrack \rm{-1}, 1 \rbrack$ \\
  \hline
  $\lambda$ & $\lbrack 15.80, 2.67\rm{e+}07 \rbrack$ & $\lbrack \rm{-}\infty, \infty \rbrack$ & $\lbrack \rm{-}\infty, \infty \rbrack$ & $\lbrack 0, 1 \rbrack$ \\
  \hline
\end{tabular}
\end{table}
\begin{table}[ht!]
\centering 
\caption{Comparison Between The Ranges Computed By Simulation, Interval Arithmetic And Affine Arithmetic With Respect To The Proposed Approach For Landscape Data With The range Of The Covariance Matrix As $\lbrack \rm{-}5.47\rm{e+}32, 6.81\rm{e+}32 \rbrack$.}\label{tab:compare3}
\begin{tabular}{p{0.4cm}|p{3.1cm}|p{1.0cm}|p{1.0cm}|p{1.0cm}}
  \hline
  Var & Simulation & IA & AA & Proposed\\
  \hline
    $A$ & $\lbrack \rm{-}5.06\rm{e+}33, 4.32\rm{e+}34 \rbrack$ & $\lbrack -\infty, \infty \rbrack$ & $\lbrack -\infty, \infty \rbrack$ & $\lbrack \rm{-1}, 1 \rbrack$ \\
  \hline
  $t$  & $\lbrack \rm{-1}, 1 \rbrack$ & $\lbrack \rm{-1}, 1 \rbrack$ & $\lbrack \rm{-1}, 1 \rbrack$ & $\lbrack \rm{-1}, 1 \rbrack$ \\
  \hline
  $cs$  & $\lbrack 0.71, 1 \rbrack$ & $\lbrack 0, 1 \rbrack$ & $\lbrack 0, 1 \rbrack$ & $\lbrack 0, 1 \rbrack$ \\
  \hline
  $sn$  & $\lbrack \rm{-1}, 1 \rbrack$ & $\lbrack \rm{-1}, 1 \rbrack$ & $\lbrack \rm{-1}, 1 \rbrack$ & $\lbrack \rm{-1}, 1 \rbrack$
 \\
  \hline
  $a$  & $\lbrack 0, 4.32\rm{e+}34 \rbrack$ & $\lbrack \rm{-}\infty, \infty \rbrack$ & $\lbrack \rm{-}\infty, \infty \rbrack$ & $\lbrack 0, 1 \rbrack$ \\
  \hline
  $b$  & $\lbrack 0, 4.32\rm{e+}34 \rbrack$ & $\lbrack \rm{-}\infty, \infty \rbrack$ & $\lbrack \rm{-}\infty, \infty \rbrack$ & $\lbrack 0, 1 \rbrack$ \\
  \hline
  $c$  & $\lbrack \rm{-}5.06\rm{e+}33, 2.12\rm{e+}33 \rbrack$ & $\lbrack \rm{-}\infty, \infty \rbrack$ & $\lbrack \rm{-}\infty, \infty \rbrack$ & $\lbrack \rm{-1}, 1 \rbrack$ \\
  \hline
  $X$ & $\lbrack \rm{-}0.932, 1 \rbrack$ & $\lbrack \rm{-}\infty, \infty \rbrack$  & $\lbrack \rm{-}\infty, \infty \rbrack$ & $\lbrack \rm{-1}, 1 \rbrack$ \\
  \hline
  $\lambda$ & $\lbrack 1.0\rm{e+}19, 4.32\rm{e+}34 \rbrack$ & $\lbrack \rm{-}\infty, \infty \rbrack$ & $\lbrack \rm{-}\infty, \infty \rbrack$ & $\lbrack 0, 1 \rbrack$ \\
  \hline
\end{tabular}
\end{table}
\begin{table}[ht!]
\centering 
\caption{Comparison Between The Integer Wordlengths Required Based On The Ranges Estimated By Simulation-Based Approach Shown In Tables~\ref{tab:compare2} And \ref{tab:compare3}}\label{tab:iwl2}
\begin{tabular}{p{0.4cm}|p{1.0cm}|p{0.9cm}}
  \hline
  Var  & AVIRIS & Landscape \\
  \hline
    $A$   & 25 & 116\\
  \hline
  $a$    & 24 & 116\\
  \hline
  $b$    & 25 & 116\\
  \hline
  $\lambda$   & 25 & 116\\
  \hline
\end{tabular}
\end{table}

\begin{table}[ht!]
\centering \caption{Signal-to-quantization-noise-ratio Of Eigenvalues Obtained In Fixed-point Arithmetic (WLs chosen are as a
general bitwidth considering the worst case) After Determining Ranges Through Proposed Approach.} 
\label{tab:sqnr}
\begin{tabular}{p{1.0cm}|p{0.5cm}|p{0.5cm}|p{0.5cm}}
\hline

 WLs   & \multicolumn{1}{p{0.8cm}|}{50 bits} & \multicolumn{1}{p{0.8cm}|}{40 bits} & \multicolumn{1}{p{0.8cm}}{32 bits} \\ \hline
Hyperion  & 176.76                  & 106.44                  & 78.03                                                                                            \\ \hline
ROSIS  & 180.13                       & 134.79                       & 74.96                                                                                             \\ \hline
Landscape & 180.65 & 122.36 & 77.18 \\ \hline
AVIRIS & 178.54 & 110.76 & 76.43 \\ \hline
EnMap & 180.67 & 130.24 & 78.36 \\ \hline
\end{tabular}
\end{table}
\begin{table*}[ht!]
\centering \caption{Mean-square-error Of PCs Obtained In Fixed-point Arithmetic (WLs chosen are as a
general bitwidth considering the worst case) After Determining Ranges Through Proposed Approach.} 
\label{tab:mse}
\begin{tabular}{p{0.37cm}|p{0.005cm}|p{0.005cm}|p{0.005cm}|p{0.005cm}|p{0.005cm}|p{0.005cm}|p{0.005cm}|p{0.005cm}|p{0.005cm}|p{0.005cm}|p{0.005cm}}
\hline
WLs & \multicolumn{4}{c|}{Hyperion}                                                                          & \multicolumn{5}{c|}{ROSIS}              & \multicolumn{2}{c}{Landscape}                                                                                         \\ \hline
    & \multicolumn{1}{p{0.43cm}|}{PC1} & \multicolumn{1}{p{0.43cm}|}{PC2} & \multicolumn{1}{p{0.43cm}|}{PC3} & \multicolumn{1}{p{0.43cm}|}{PC4} & \multicolumn{1}{p{0.43cm}|}{PC1} & \multicolumn{1}{p{0.43cm}|}{PC2} & \multicolumn{1}{p{0.43cm}|}{PC3} & \multicolumn{1}{p{0.43cm}|}{PC4} & \multicolumn{1}{p{0.43cm}|}{PC5} & \multicolumn{1}{p{0.43cm}|}{PC1} & \multicolumn{1}{p{0.43cm}}{PC2} \\ \hline
32  & 8.1\rm{e-}7                  & 4.9\rm{e-}7                  & 3.7\rm{e-}6                  & 4.3\rm{e-}6                  & 0                       & 1.8\rm{e-}7                  & 4.5\rm{e-}6                  & 6.5\rm{e-}6                  & 2.4\rm{e-}6       & 1.2\rm{e-}9 & 1.1\rm{e-}8           \\ \hline
40  & 0                       & 0                       & 8.5\rm{e-}7                  & 1.9\rm{e-}7                  & 0                       & 0                       & 0                       & 0                       & 0           & 1.8\rm{e-}10 & 7.4\rm{e-}11            \\ \hline
50  & 0                       & 0                       & 0                       & 0                       & 0                       & 0                       & 0                       & 0                       & 0          & 0 & 0            \\ \hline
\end{tabular}
\end{table*}   
\indent Signal-to-quantization-noise-ratio (SQNR) is chosen as an error measure to evaluate the accuracy of the proposed method \cite{7:1:137}, \cite{dm2012}. It is given by
\begin{equation}
\label{eq:sqnr}
\begin{split}
SQNR=10\log_{10}({\textrm{E}({|{\lambda}_{float}|}^2)})/({\textrm{E}({|{\lambda}_{float}-{\lambda}_{fixed}|}^2)}),
\end{split}
\end{equation}
where ${\lambda}_{float}$ and ${\lambda}_{fixed}$ are the
eigenvalues obtained from double precision floating-point and fixed-point
implementations. SQNR of the eigenvalues obtained through the proposed fixed-point design is shown in Table~\ref{tab:sqnr}. In Table~\ref{tab:sqnr}, we observe high magnitudes of SQNR which exhibit that the set of ranges obtained according to Theorem 2 are sufficient for avoiding overflow for any input matrix. For data sets like Landscape, where the the range is exorbitant resulting in large IWLs (Table~\ref{tab:iwl2}), wordlengths like 50, 40 or 32 bits would never fit. In such cases, with the proposed approach it was possible to fit all the variables within 32 bit wordlength and obtain a high value of SQNR. A common measure to compare two images is mean-square-error (MSE) \cite{wang2009mean}. MSE between PCA images of fixed-point implementations with various WLs after derving the ranges through proposed approach are shown in Table~\ref{tab:mse}. The required number of PCs for Hyperion, ROSIS
and Landscape are 4, 5 and 2 respectively. The number of PCs explaining 99.0$\%$ variance in case of AVIRIS and EnMap are relatively higher. Therefore, the Table~\ref{tab:mse} only exhibits the results of Hyperion, ROSIS and Landscape. However, similar results were obtained for AVIRIS and EnMap. 
We observe that the MSE values are negligibly small which signify that the ranges obtained through the proposed approach are absolutely robust. Thus, the error metrics (SQNR and MSE) imply that the number of integer bits derived using the proposed approach is sufficient for avoiding overflow. After deriving the proper ranges through the proposed approach, the fixed-point design is synthesized on Xilinx Virtex 7 XC7VX485 FPGA for different WLs through Vivado high-level synthesis (HLS) design tool \cite{design2013high}. We have used SystemC (mimics hardware description language VHDL and Verilog) to develop the fixed-point code \cite{ieee2005ieee}. Using the HLS tool, the SystemC fixed-point code is transformed into a hardware IP (intellectual property) described in Verilog. We compare the resource utilization of simulation approach with respect to the proposed approach (for the same level of accuracy) through the test hyperspectral data sets. The comparative study is illustrated in Table~\ref{tab:comdynpro}. There is a noteworthy difference in the hardware resources. The hardware resources in case of simulation approach are considerably large compared to the resources used in case of the proposed approach. For the sake of maintaining the same level of accuracy (SQNR, MSE) as the proposed method, the simulation approach uses 50 bit wordlength. \\
\indent The proposed method also produces robust and tight analytical bounds for variables of singular value decomposition algorithm \cite{siam}.

\begin{table}[ht!]
\centering 
\caption{Comparison Between Hardware Cost ($\%$) Of Fixed-point Jacobi Algorithm After Determining Ranges Through Proposed And Simulation Approaches.} 
\label{tab:comdynpro}
\begin{tabular}{p{0.32cm}|p{0.005cm}|p{0.005cm}|p{0.005cm}|p{0.005cm}|p{0.005cm}|p{0.005cm}|p{0.005cm}|p{0.005cm}|p{0.005cm}}
\hline
     \multicolumn{10}{c}{Proposed}                                                                                                                                                                                 \\ \hline
WL & \multicolumn{3}{c|}{Hyperion}                                                                          & \multicolumn{3}{c|}{ROSIS}             &  \multicolumn{3}{c}{AVIRIS}                                                                                          \\ \hline
    & \multicolumn{1}{p{0.32cm}|}{FF} & \multicolumn{1}{p{0.53cm}|}{LUTs} & \multicolumn{1}{p{0.53cm}|}{Power} & \multicolumn{1}{p{0.32cm}|}{FF} & \multicolumn{1}{p{0.53cm}|}{LUTs} & \multicolumn{1}{p{0.53cm}|}{Power}  & \multicolumn{1}{p{0.32cm}|}{FF} & \multicolumn{1}{p{0.53cm}|}{LUTs} & \multicolumn{1}{p{0.53cm}}{Power} \\ \hline
32  & 1.62                  & 6.29                  & 0.413                                    & 1.62                       & 6.32                  & 0.42    & 1.63                  & 6.51                  & 0.45                                \\ \hline
\multicolumn{10}{c}{Simulation}                                                                                                                                                                                 \\ \hline
WL & \multicolumn{3}{c|}{Hyperion}                                                                          & \multicolumn{3}{c|}{ROSIS}     &  \multicolumn{3}{c}{AVIRIS}                                                                                                  \\ \hline
    & \multicolumn{1}{p{0.32cm}|}{FF} & \multicolumn{1}{p{0.53cm}|}{LUTs} & \multicolumn{1}{p{0.53cm}|}{Power} & \multicolumn{1}{p{0.32cm}|}{FF} & \multicolumn{1}{p{0.53cm}|}{LUTs} & \multicolumn{1}{p{0.53cm}|}{Power}  &  \multicolumn{1}{p{0.32cm}|}{FF} & \multicolumn{1}{p{0.53cm}|}{LUTs} & \multicolumn{1}{p{0.53cm}}{Power} \\ \hline
50  & 8                  & 23                  & 2.59                             & 8                       & 23                  & 2.64        & 8                       & 23                  & 2.64                           \\ \hline
\end{tabular}
\end{table}

\section{Conclusion}
\label{sec:conclu}
In this paper, we bring out the problem of integer bit-width allocation for the variables of eigenvalue decomposition algorithm. We highlight the issues of the existing range estimation methods in the context of EVD. Integer bit-width allocation is an essential step in fixed-point hardware design. In light of the significance of this step, this paper introduces an analytical method based on vector and matrix norm properties together with a scaling procedure to produce robust and tight bounds. Through some hyperspectral data sets, we demonstrate the efficacy of the proposed method in dealing with the issues associated with existing methods. SQNR and MSE values show that the ranges derived using the proposed approach are sufficient for avoiding overflow in case of any input matrix.  There are many other numerical linear algebra algorithms which can benefit from the
proposed method like QR factorization, power method for finding largest eigenvalue,
bisection method for finding eigenvalues of a symmetric tridiagonal matrix, QR iteration,
Arnoldi method for transforming a non-symmetric matrix into an upper Hessenberg matrix
and LU factorization and Cholesky factorization. \\
\indent Dealing with the precision problem will be a scope for the future work.



\bibliographystyle{IEEEtran}
\bibliography{paper}

\clearpage
\end{document}